\documentclass[10pt,a4paper]{amsart}
\usepackage{amsmath,amsthm,amssymb,latexsym}
\usepackage{enumerate}
\usepackage{graphicx}

\newtheorem{theorem}{Theorem}
\newtheorem{lemma}[theorem]{Lemma}
\newtheorem{corollary}[theorem]{Corollary}
\newtheorem{proposition}[theorem]{Proposition}
\newtheorem{exa}[theorem]{Example}
\newtheorem{rema}[theorem]{Remark}

\newtheorem{defi}[theorem]{Definition}

\newtheorem{pro}[theorem]{Open Problem}

\newenvironment{remark}{\begin{rema}\rm}{\end{rema}}
\newenvironment{example}{\begin{exa}\rm}{\end{exa}}
\newenvironment{problem}{\begin{pro}\rm}{\end{pro}}
\numberwithin{equation}{section} \numberwithin{theorem}{section}

\title{Cayley automaton semigroups}
\author{Victor Maltcev}

\date{}

\begin{document}

\maketitle

\begin{center}
Mathematical Institute, University of St Andrews\\
St Andrews, Fife KY16 9SS, Scotland,
\texttt{victor.maltcev@gmail.com}
\end{center}

\begin{abstract}
In this paper we characterize when a Cayley automaton semigroup is a
group, is trivial, is finite, is free, is a left zero semigroup, or
is a right zero semigroup.
\end{abstract}


\section{Introduction \& Main Results}

In a bright work~\cite{SS} Silva \& Steinberg introduce the notion
of a \emph{Cayley automaton} of a semigroup: having a finite
semigroup $S$, let $\mathcal{C}(S)$ be the automaton with state set
$S$ and alphabet set $S$, obtained from the Cayley graph of $S$ by
letting the output symbol on the arc leading from $s$ and labeled by
$t$ to be $st$:

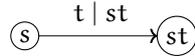
\begin{figure}[htp]
\centering
\includegraphics{cm_1auto.6}
\end{figure}
Every state from $\mathcal{C}(S)$ can be viewed as a transformation
on the set of all infinite sequences $S^{\infty}$. The semigroup
$\mathbf{C}(S)$ generated by all such transformations, associated to
the states of $\mathcal{C}(S)$, is obviously the automaton semigroup
generated by the automaton $\mathcal{C}(S)$ (in the sense
of~\cite{GNS} and~\cite{N-book}). Silva \& Steinberg prove the
following

\begin{theorem}\label{th:ss-groups}
Let $G$ be a finite non-trivial group. Then $\mathbf{C}(G)$ is a
free semigroup of rank $|G|$.
\end{theorem}

Under slightly other prospective, Cayley automaton semigroups,
derived from monoids, appeared in a work by Mintz~\cite{Mi}. In
particular, he proves that if $S$ is a finite $\mathcal{H}$-trivial
monoid, then $\mathbf{C}(S)$ is a finite $\mathcal{H}$-trivial
semigroup.

The aim of this paper is the following theorems and propositions:

\begin{theorem}\label{th:whenisgroup}
For a finite semigroup $S$, the following statements are equivalent:
\begin{enumerate}
\item
$\mathbf{C}(S)$ is a group.
\item
$\mathbf{C}(S)$ is trivial.
\item
$S$ is an inflation of a right zero semigroup by null semigroups.
\end{enumerate}
\end{theorem}

We prove Theorem~\ref{th:whenisgroup} in Section~\ref{sec:group}.

\begin{theorem}\label{th:finite}
Let $S$ be a finite semigroup. Then $\mathbf{C}(S)$ is finite if and
only if $S$ is $\mathcal{H}$-trivial.
\end{theorem}

The proof of sufficiency of Theorem~\ref{th:finite} in the case when
$S$ is a monoid was proved in~\cite{Mi}. The analogous
characterization for the so-called \emph{dual Cayley automaton
semigroups} can be found in~\cite{M}. The proof of
Theorem~\ref{th:finite} is contained in Section~\ref{sec:finite}.

Making use of Theorem~\ref{th:finite}, in Sections~\ref{sec:free}
and~\ref{sec:left-right} we will prove the following three
propositions:

\begin{proposition}\label{pr:free}
Let $S$ be a finite semigroup. Then $\mathbf{C}(S)$ is a free
semigroup if and only if the minimal ideal $K$ of $S$ consists of a
single $\mathcal{R}$-class, in which every $\mathcal{H}$-class is
not a singleton, and there exists $k\in K$ such that $st=skt$ for
all $s,t\in S$.
\end{proposition}

\begin{proposition}\label{pr:right-zero}
Let $S$ be a finite semigroup. Then $\mathbf{C}(S)$ is a right zero
semigroup if and only if $abc=ac$ for all $a,b,c\in S$.
\end{proposition}

\begin{proposition}\label{pr:left-zero}
Let $S$ be a finite semigroup. Then $\mathbf{C}(S)$ is a left zero
semigroup if and only if $S^2$ is the minimal ideal of $S$ and if
this ideal forms a right zero semigroup.
\end{proposition}

In the final Section~\ref{sec:further} we will discuss our main
results and their corollaries. Before we start proving our
statements in the next section we give all necessary notation and
lemmas needed for the proofs.


\section{Auxiliary Lemmas}\label{sec:aux}

Let $S=\{s_1,\ldots,s_n\}$ be a finite semigroup. In order to avoid
confusion, we will denote the states in $\mathcal{C}(S)$ by an overline: $s$ is a symbol,
and $\overline{s}$ is a state. The sequences from $S^{\infty}$ can
be vied as paths in the infinite $S$-rooted tree. So,
following~\cite{N-book}, we will think about $\overline{s}$ in terms
of the \emph{wreath recursion}:
\begin{equation*}
\overline{s}=\lambda_s(\overline{ss_1},\ldots,\overline{ss_n}),
\end{equation*}
where $\lambda_s:S\to S$, defined by $x\mapsto sx$, corresponds to
the action of $\overline{s}$ on the first level of the $S$-tree.
Notice that $\lambda_s\lambda_t=\lambda_{ts}$ for all $s,t\in S$.
Hence
\begin{equation*}
\overline{s}\cdot\overline{t}=\lambda_{ts}(\overline{ss_1}\cdot\overline{tss_1},\ldots,
\overline{ss_n}\cdot\overline{tss_n}).
\end{equation*}
Iterating this formula one obtains that when automaton
$\mathcal{C}(S)$ is in state $\overline{a_1}\cdots\overline{a_k}$
and reads a symbol $x$, it moves to the state
\begin{equation*}
q(\overline{a_1}\cdots\overline{a_k},x)=\overline{a_1x}\cdot\overline{a_2a_1x}\cdots\overline{a_k\cdots
a_1x}.
\end{equation*}
The transformation $\tau(\overline{a_1}\cdots\overline{a_k})$ on the
set $S$, correspondent to the action of
$\overline{a_1}\cdots\overline{a_k}$ on the first level of the
$S$-tree is obviously $\lambda_{a_k\cdots a_1}$.

\begin{lemma}\label{lm:lambda}
Let $S$ be a finite semigroup. Then for all $s,t\in S$,
$\overline{s}=\overline{t}$ in $\mathbf{C}(S)$ if and only if
$\lambda_s=\lambda_t$.
\end{lemma}

\begin{proof}
Let $s,t\in S$. Then $\overline{s}=\overline{t}$ if and only if
$\lambda_s=\lambda_t$ and $\overline{sx}=\overline{tx}$ for all
$x\in S$. Recursing this, we obtain that $\overline{s}=\overline{t}$
if and only if $\lambda_s=\lambda_t$ and $\lambda_{sx}=\lambda_{tx}$
for all $x\in S$. It remains to notice that $\lambda_s=\lambda_t$
implies $\lambda_{sx}=\lambda_{tx}$ for all $x\in S$.
\end{proof}

Now we calculate Cayley automaton semigroups of special type
semigroups:

\begin{lemma}\label{lm:left-zero}
Let $L$ be a finite left zero semigroup. Then $\mathbf{C}(L)$ is a
right zero semigroup with $|L|$ elements.
\end{lemma}

\begin{proof}
Suppose $\mathcal{C}(L)$ is in state $\overline{s}$ and reads symbol
$t$. Then, by the definition of $\mathcal{C}(L)$, it outputs $s$ and
moves to the same state $\overline{s}$. Thus
$\alpha\cdot\overline{s}=s^{\infty}$ for all $\alpha\in L^{\infty}$.
Hence for any $s,t\in S$ and $\alpha\in L^{\infty}$:
\begin{equation*}
\alpha\cdot(\overline{s}\cdot\overline{t})=s^{\infty}\cdot\overline{t}
=t^{\infty}=\alpha\cdot\overline{t},
\end{equation*}
and so $\overline{s}\cdot\overline{t}=\overline{t}$. It remains to
note that, by Lemma~\ref{lm:lambda}, if $s\neq t$, then
$\overline{s}\neq\overline{t}$.
\end{proof}

\begin{lemma}\label{lm:direct}
Let $S$ be a finite semigroup and let $R$ be a finite right zero
semigroup. Then $\mathbf{C}(S\times R)\cong\mathbf{C}(S)$.
\end{lemma}

\begin{proof}
Let $s\in S$ and $r,t\in R$. Then it follows from
Lemma~\ref{lm:lambda} that $\overline{(s,r)}=\overline{(s,t)}$ in
$\mathbf{C}(S\times R)$. Hence $\mathbf{C}(S\times R)$ coincides
with $T=\langle\overline{(s,r_0)}:s\in S\rangle$ for any fixed $r_0\in R$. It
is now easy to check that $\overline{(s,r_0)}\mapsto\overline{s}$
gives rise to an isomorphism from $T$ onto $\mathbf{C}(S)$.
\end{proof}

\begin{corollary}
Let $R$ be a finite right zero semigroup. Then $\mathbf{C}(R)$ is
trivial.
\end{corollary}

The proof of the last lemma of this section is easy, it can be found
in~\cite{M}.

\begin{lemma}\label{lm:D-class}
Let $S$ be a finite semigroup and let $a,b\in S$. If all $a$, $b$
and $ab$ belong to the same $\mathcal{D}$-class of $S$ then $ab\in
R_{a}\cap L_{b}$.
\end{lemma}


\section{Proof of Theorem~\ref{th:whenisgroup}}\label{sec:group}

We recall that a semigroup $S$ is an inflation of a right zero
semigroup $T$ by null semigroups if $T\leq S$ and $S$ can be
partitioned into disjoint subsets $S_t$ (for each $t\in T$) such
that $t\in S_t$ and $S_uS_t=\{t\}$ for all $t,u\in S$.

\begin{proof}[Proof of Theorem~\ref{th:whenisgroup}]
The proof follows via the chain
$(1)\Rightarrow(3)\Rightarrow(2)\Rightarrow(1)$.

$(2)\Rightarrow(1)$ is clear.

$(3)\Rightarrow(2)$. Let $S$ be an inflation of a right zero
semigroup $T$. Then for all $s,t,x\in S$, we have $sx=tx$. Hence
$\lambda_s=\lambda_t$ and so, by Lemma~\ref{lm:lambda},
$\overline{s}=\overline{t}$ for all $s,t\in S$. It remains to prove
that for any fixed $s\in S$, the element $\overline{s}$ is an
idempotent. We have
$\overline{s}=\lambda_s(\overline{s},\ldots,\overline{s})$ and
$\overline{s}^2=\tau(\overline{s}^2)(\overline{s}^2,\ldots,\overline{s}^2)$.
Thus it suffices to prove that $\tau(\overline{s}^2)=\lambda_s$.
This holds since $s^2x=sx$ for all $x\in S$.

$(1)\Rightarrow(3)$. We will prove by induction on $|S|$ that if
$\mathbf{C}(S)$ is a group, then $S$ is an inflation of a right zero
semigroup by null semigroups. The base case $|S|=1$ is trivial. So
suppose the implication holds for all semigroups of cardinality
$<|S|$ and that $\mathbf{C}(S)$ is a group.

Let $\mathcal{T}_S$ be the transformation semigroup on $S$. The
subsemigroup $\langle\lambda_s:s\in S\rangle$ in $\mathcal{T}_S$ is
a homomorphic image of the group $\mathbf{C}(S)$ and so is a group.
This implies that the images and kernels of the mappings $\lambda_s$
must coincide. These conditions can be translated as
\begin{itemize}
\item
$sS=tS$ for all $s,t\in S$ and
\item
$sx=sy$ if and only if $tx=ty$, for all $s,t,x,y\in S$.
\end{itemize}
Notice that the condition that $sS=tS$ for all $s,t\in S$ is
equivalent to $sS=S^2$. The rest of the proof depends on whether
$S^2=S$ or not.

\vspace{\baselineskip}

{\bf\noindent Case 1: $S^2=S$.}

Then, by the observations above, $sS=S^2=S$ for all $s\in S$. So
each $s\in S$, acting via left-multiplication, permutes $S$. Then
for any $s\in S$, some power of $s$ is a left identity $e$ for $S$.
Then for all $s,x,y\in S$, the condition $sx=sy$ implies
$x=ex=ey=y$. Hence $S$ is left cancellative. The condition that
$sS=S$ for all $s\in S$ implies that $S$ is right simple.

Therefore $S$ is a right group and so $S=G\times R$ for some group
$G$ and a right zero semigroup $R$, see~\cite{CP}. By
Lemma~\ref{lm:direct}, we have $\mathbf{C}(G\times
R)=\mathbf{C}(G)$. Since the Cayley automaton semigroup over every non-trivial
group is a free semigroup, $G$ must trivial. Then $S\cong R$ is a
right zero semigroup and so (3) holds.

\vspace{\baselineskip}

{\bf\noindent Case 2: $S^2\neq S$.}

The condition of the case means that $S$ contains indecomposable
elements.

Recall that the kernels of all the mappings $\lambda_s$ coincide.
Partition $S$ into these kernel classes $A_1,\ldots,A_k$ and notice
that, for every $s\in S$, the equality $sx=sy$ holds if and only if
$x$ and $y$ come from the same class. Furthermore, since the
mappings $\lambda_s$ generate a subgroup in $\mathcal{T}_S$, it
follows that every kernel class $A_i$ contains an image point, which
must of course be an element of $S^2$ (the image of every
$\lambda_s$ is $S^2$).

The remainder of the proof we will work out in two subcases:

\vspace{\baselineskip}

{\bf\noindent Subcase a: for all $a\in S\setminus S^2$ there exists
$x\in S\setminus\{a\}$ with
$x(S\setminus\{a\})\neq(S\setminus\{a\})(S\setminus\{a\})$.}

Consider an arbitrary $a\in S\setminus S^2$ and find the
corresponding $x\in S\setminus\{a\}$. That
$x(S\setminus\{a\})\neq(S\setminus\{a\})(S\setminus\{a\})$ means
that there exists an element $uv\notin x(S\setminus\{a\})$ where
$u,v\in S\setminus\{a\}$. Since $uS=xS$, there exists some $b\in S$
with $uv=xb$. Obviously then $b=a$. Hence $xa\notin
x(S\setminus\{a\})$. That is, $xa\neq xy$ for all $y\in
S\setminus\{a\}$. This is equivalent to that $sa\neq sy$ for all
$s\in S$ and $y\in S\setminus\{a\}$. The kernel class $A$ that
contains $a$ is not a singleton, for it must contain an element from
$S^2$ and $a$ itself is indecomposable. Take an arbitrary $c\in
A\setminus\{a\}$. Then $sa=sc$ for any $s\in S$, a contradiction.

\vspace{\baselineskip}

{\bf\noindent Subcase b: There exists $a\in S\setminus S^2$ such
that for all $x\in S\setminus\{a\}$ there holds
$x(S\setminus\{a\})=(S\setminus\{a\})(S\setminus\{a\})$.}

Fix such an $a$. Obviously $T=S\setminus\{a\}$ is a subsemigroup of
$S$.

We will show now that $\mathbf{C}(T)$ is a homomorphic image of
$\mathbf{C}(S)$ and thus that $\mathbf{C}(T)$ is a group. Let $I$ be
the minimal ideal in $S$. Then $I$ is simple and so, being finite,
is completely simple. Hence $I$ is a Rees matrix semigroup. Take now
an arbitrary $i\in I$. Since $I$ is a Rees matrix semigroup, there
exists $e\in I$ such that $ei=i$. Then
$\tau(\overline{i}\cdot\overline{a})=\lambda_{ai}=
\lambda_{aei}=\tau(\overline{i}\cdot\overline{ae})$ and
\begin{equation*}
q(\overline{i}\cdot\overline{a},x)=\overline{ix}\cdot\overline{aix}=
\overline{ix}\cdot\overline{aeix}=q(\overline{i}\cdot\overline{ae},x)
\end{equation*}
for all $x\in S$. Thus
$\overline{i}\cdot\overline{a}=\overline{i}\cdot\overline{ae}$ in
$\mathbf{C}(S)$. Since $\mathbf{C}(S)$ is a group, we derive now
that $\overline{a}=\overline{ae}$. Since $e\in I$, the element $ae$
must lie in $I\subseteq T$. Hence $\overline{S}=\overline{T}$ in
$\mathbf{C}(S)$. Restricting the action of the states from
$\mathcal{C}(S)$ to $T^{\ast}$ yields the automaton
$\mathcal{C}(T)$. Therefore $\mathbf{C}(T)$ is a homomorphic image
of $\mathbf{C}(S)$, and as so is a group.

So, by the induction hypothesis, $T$ is an inflation of a right zero
semigroup by null semigroups. Suppose without loss of generality
that $a\in A_k$. Then $A_1,\ldots,A_{k-1},A_k\setminus\{a\}$ are the
correspondent null semigroups from $T$. For each $i$, let $e_i\in
A_i$ be the right zero in $A_i$. Then $S^2=\{e_1,\ldots,e_k\}$. In
particular, $e_k\neq a$ since $a$ is indecomposable. Take now
$a_i\in A_i$ and $a_j\in A_j$. Recall that $sx=sy$ as soon as $x$
and $y$ are from the same kernel class.
\begin{enumerate}
\item
If $a_i\neq a$ and $a_j\neq a$, then $a_ia_j=e_ie_j=e_j$.
\item
If $a_i\neq a$ and $a_j=a$, then $a_ia_j=a_ia=a_ie_k=e_k$.
\item
Let $a_i=a$ and $a_j\neq a$. Let $aa_j=e_m$ for some $m$. Then
$e_m=e_m^2=e_ma a_j$. Since $e_ma\in T$, it follows that
$e_maa_j=e_j$ and so $e_m=e_maa_j=e_j$. Hence $a_ia_j=e_j$.
\item
If $a_i=a_j=a$, then $a_ia_j=a^2=ae_k=e_k$.
\end{enumerate}

Thus $S$ is an inflation of a right zero semigroup
$\{e_1,\ldots,e_k\}$ and the induction step is established.
\end{proof}

\section{Proof of Theorem~\ref{th:finite}}\label{sec:finite}

\begin{proof}[Proof of Theorem~\ref{th:finite}]
$(\Rightarrow)$. Suppose that $\mathbf{C}(S)$ is finite. Take any
$\mathcal{H}$-class $H$ in $S$. With the seek of a contradiction,
suppose that $|H|>1$. Let $T=\{t\in S:tH\subseteq H\}$. Then for
every $t\in T$, by~\cite[Lemma~2.21]{CP}, the mapping
$\gamma_t:h\mapsto th$, $h\in H$, is a bijection of $H$ onto itself.
The set of all these bijections forms the so-called \emph{dual
Sch\"{u}tzenberger group} $\Gamma^{\ast}(H)$ of $H$.
By~\cite[Theorem~2.22]{CP} we have $|\Gamma^{\ast}(H)|=|H|$. Let
$\Delta(H)$ be the dual group of $\Gamma^{\ast}(H)$: that is has the
same underlying set as $\Gamma^{\ast}(H)$ but in $\Delta(H)$ we have
$\gamma_x\circ\gamma_y=\gamma_{xy}$ for all $x,y\in T$..

Take arbitrary $t_1,\ldots,t_k\in T$. Then for all $x\in H$:
\begin{equation*}
q(\overline{t_1}\cdots\overline{t_k},x)=
\overline{t_1x}\cdot\overline{t_2t_1x}\cdots\overline{t_k\cdots
t_1x}.
\end{equation*}
We also have
$\tau(\overline{t_1}\cdots\overline{t_k})=\tau(\overline{t_k\cdots
t_1})$.

Take now
$\gamma_{t_1},\ldots,\gamma_{t_k},\gamma_x\in\mathcal{C}(\Delta(H))$.
Then
\begin{equation*}
q(\overline{\gamma_{t_1}}\cdots\overline{\gamma_{t_k}},\gamma_{x})=
\overline{\gamma_{t_1}\circ\gamma_{x}}\cdot\overline{\gamma_{t_2}\circ
\gamma_{t_1}\circ\gamma_{x}}\cdots\overline{\gamma_{t_k}\circ\cdots\circ
\gamma_{t_1}\circ\gamma_{x}}
\end{equation*}
and $\tau(\overline{\gamma_{t_1}}\cdots\overline{\gamma_{t_k}})=
\tau(\overline{\gamma_{t_k}\circ\cdots\circ\gamma_{t_1}})=\tau(\overline{\gamma_{t_k\cdots
t_1}})$.

Take $t\in T$ and consider the restriction of $\overline{t}$ to
$H^{\ast}$. From the very definition of $\Gamma^{\ast}(H)$, it now
follows that the mapping
$\overline{t}\upharpoonright_{H^{\ast}}\mapsto\overline{\gamma_t}$
gives rise to a well-defined homomorphism from
$\langle\overline{t}\upharpoonright_{H^{\ast}}:t\in T\rangle$ onto
$\mathbf{C}(\Delta(H))$. It means that $\langle\overline{T}\rangle$
has a free semigroup on $|\Delta(H)|=|\Gamma^{\ast}(H)|=|H|$ points,
as a homomorphic image, and so $\langle\overline{T}\rangle$ is
infinite. Thus $\mathbf{C}(S)$ is infinite, a contradiction.

$(\Leftarrow)$. We will prove by induction on $|S|$ that if $S$ is
$\mathcal{H}$-trivial then $\mathbf{C}(S)$ is finite. The base case
$|S|=1$ is obvious. Assume that we have proved this for all
$\mathcal{H}$-trivial semigroups of size $\leq n$. Take now any
$\mathcal{H}$-trivial semigroup $S$ with $|S|=n+1$. Let $M$ be the
set of all maximal $\mathcal{D}$-classes from $S$ and let $I$ be the
complement of all these $\mathcal{D}$-classes in $S$. If $I$ is
empty then $M$ consists only of one $\mathcal{D}$-class and then $S$
is simple. Since it is $\mathcal{H}$-trivial we have that $S=L\times
R$ is a rectangular band, where $L$ is some left zero semigroup and
$R$ is some right zero semigroup. Combining
Lemmas~\ref{lm:left-zero} and~\ref{lm:direct}, we
have that $\mathbf{C}(S)$ is a right zero semigroup on $|L|$ points
and so is finite. So in the remainder of the proof we may assume
that $I\neq\varnothing$. Notice that $I$ is an ideal in $S$.

\vspace{\baselineskip}

{\bf\noindent Step 1: $\mathbf{\langle\overline{I}\rangle}$ is
finite.}

It suffices to prove that there are finitely many products
$\overline{i}\cdot\overline{i_1}\cdots\overline{i_k}\in\langle\overline{I}\rangle$
for any fixed $i\in I$. We have that
$\overline{i}\cdot\overline{i_1}\cdots\overline{i_k}$ and
$\overline{i}\cdot\overline{j_1}\cdots\overline{j_n}$ are
distinct if and only if the restrictions of
$\overline{i_1}\cdots\overline{i_k}$ and
$\overline{j_1}\cdots\overline{j_n}$ on
$S^{\infty}\overline{i}$ coincide. Notice that
$S^{\infty}\overline{i}\subseteq I^{\infty}$. Obviously
$\overline{i_1}\cdots\overline{i_k}$ and
$\overline{j_1}\cdots\overline{j_n}$ act on $I^{\infty}$
in the same way as the correspondent products from $\mathbf{C}(I)$
do. Now the claim of Step 1 follows from the induction hypothesis.

\vspace{\baselineskip}

{\bf\noindent Step 2:
$\mathbf{\overline{I}\langle\overline{S}\rangle}$ is finite.}

Take a typical element
$\overline{i}\cdot\overline{a_1}\cdots\overline{a_k}\in
\overline{I}\langle\overline{S}\rangle$. Then for all $x\in S$, we
have
\begin{equation*}
q(\overline{i}\cdot\overline{a_1}\cdots\overline{a_k},x)=
\overline{ix}\cdot\overline{a_1ix}\cdots\overline{a_k\cdots
a_1ix}.
\end{equation*}
Having that $I$ is an ideal in $S$, we deduce that
$|\overline{I}\langle\overline{S}\rangle|\leq
|I|\cdot|\langle\overline{I}\rangle|^{|S|}$.

\vspace{\baselineskip}

{\bf\noindent Step 3:
$\mathbf{\langle\overline{S}\setminus\overline{I}\rangle}$ is
finite.}

We need to prove that there are only finitely many distinct products
$\overline{a_1}\cdots\overline{a_k}$, where all
$a_1,\ldots,a_k$ lie in $S\setminus I$. Take such $a_1,\ldots,a_k$.
We have that
\begin{equation}\label{eq:equ}
q_x=q(\overline{a_1}\cdots\overline{a_k},x)=\overline{a_1x}\cdot\overline{a_2a_1x}
\cdots\overline{a_k\cdots a_1x}.
\end{equation}
Obviously, to prove Step 3, it suffices to establish that there only
finitely many such expressions~\eqref{eq:equ}. It follows
immediately from Step 2, that there are finitely many such
expressions with $a_1x\in I$.

Denote by $\mathcal{D}_M$ the restriction of $\mathcal{D}$ to
$S\setminus I$. Notice that if $(a_i\cdots a_1x,a_i\cdots
a_1)\notin\mathcal{D}_M$, for some $i$, then $a_i\cdots a_1x\in I$.
Indeed, if $a_i\cdots a_1x\in S\setminus I$, then $a_i\cdots a_1\in
S\setminus I$ and $x\in S\setminus I$. But since $D_{a_i\cdots
a_1x}\leq D_{a_i\cdots a_1}$ and $D_{a_i\cdots a_1x}\leq D_{x}$, we
now have that $a_i\cdots a_1\mathcal{D}x\mathcal{D}a_i\cdots a_1x$
(all three elements $a_i\cdots a_1$, $x$ and $a_i\cdots a_1x$ lie in
maximal $\mathcal{D}$-classes). In particular, if
$(a_1x,a_1)\notin\mathcal{D}_M$ then $a_1x\in I$ and so
$q_x\in\overline{I}\langle\overline{S}\rangle\cup\overline{I}$.

From the above it follows that it suffices to prove that for a fixed
$x\in S$ with $a_1x\mathcal{D}a_1$ (which is the same as
$a_1x\mathcal{D}_M a_1$), there exist only finitely many expressions
$q_x$. Let $m$ be the maximum number such that
\begin{equation}\label{eq:yes-we-restrict-further}
a_i\cdots a_1\mathcal{D}_M a_i\cdots a_1x,\quad 1\leq i\leq m.
\end{equation}
Since $a_{m+1}\cdots a_1x\in I$, by Step 2, it suffices to
prove that there are finitely many products
$\overline{a_1}\cdots\overline{a_m}$
with~\eqref{eq:yes-we-restrict-further}. Consider such one. We have
$a_i\cdots a_1\mathcal{D}_M x$ for all $i\leq m$. In particular then
we have $a_m\cdots a_1,\ldots,a_1$ are all from the same
$\mathcal{D}$-class (in $M$). This implies
$a_1\mathcal{D}a_2\mathcal{D}\cdots\mathcal{D}a_m$. Then, by
Lemma~\ref{lm:D-class},
$a_1\mathcal{L}a_2a_1\mathcal{L}\cdots\mathcal{L}a_m\cdots a_1$.
Since $\mathcal{L}$ is a right congruence, we then have that
$a_1x\mathcal{L}a_2a_1x\mathcal{L}\cdots\mathcal{L}a_m\cdots a_1x$.

Recall that it is enough to prove that there are only finitely many
products $\overline{a_1x}\cdot\overline{a_2a_1x}\cdots\overline{a_m\cdots a_1x}$
with~\eqref{eq:yes-we-restrict-further}.

Thus it suffices to prove that there exist only finitely many
different products $\overline{a_1}\cdots\overline{a_n}$
such that $a_1,\ldots,a_n$ all come from the same
$\mathcal{L}$-class inside a $\mathcal{D}$-class from $M$. Obviously
it suffices to prove this for a fixed $\mathcal{D}$-class $D$ in
$M$.

So, we need to prove that the set
\begin{equation*}
P=\{\overline{a_1}\cdots\overline{a_n}:
a_1\mathcal{L}a_2\mathcal{L}\cdots\mathcal{L}a_n,~a_1\in D\}
\end{equation*}
is finite and then Step 3 is established.

Again we consider for $x\in S$ the elements
$q_x=\overline{a_1x}\cdot\overline{a_2a_1x}
\cdots\overline{a_n\cdots a_1x}$, for
$\overline{a_1}\cdots\overline{a_n}\in P$. It suffices to
prove that there are finitely many such $q_x$-s.

If $a_1x\notin D$, then
$q_x\in\overline{I}\langle\overline{S}\rangle\cup\overline{I}$. So
we may assume that $x\in S$ is such that $a_1x\in D$. Find the
maximum $m$ such that $a_i\cdots a_1,a_i\cdots a_1x\in D$ for all
$i\leq m$. Since $a_{m+1}\cdots a_1x\in I$, as before, it suffices
to prove that there are finitely many products
$\overline{a_1x}\cdot\overline{a_2a_1x}
\cdots\overline{a_m\cdots a_1x}$.

We have $a_2a_1,\ldots,a_{m}a_{m-1}\in D$. Now, for all $j\leq m-1$,
$a_{j+1}\mathcal{D}a_{j+1}a_{j}\mathcal{D}a_{j}$ and so by
Lemma~\ref{lm:D-class} we have that $a_{j+1}a_j\in R_{a_{j+1}}\cap
L_{a_j}$. By Clifford-Miller Theorem we obtain that $L_{a_{j+1}}\cap
R_{a_j}$ contains an idempotent. Since $S$ is $\mathcal{H}$-trivial
and $a_1\mathcal{L}\cdots\mathcal{L}a_m$ we have that $a_{j}$ is an
idempotent and $a_{j+1}a_j=a_{j+1}$. It follows that
$\{a_1,\ldots,a_{m-1}\}$ forms a left zero subsemigroup in $S$.

By the token as above, we have that it suffices to show that there
are finitely many products $Q\subseteq P$ of the type
$\overline{a_1}\cdots\overline{a_n}$ with
$a_1,\ldots,a_n\in D$ and all of them lying in the same
$\mathcal{L}$-class and forming a left zero semigroup.

We will prove by induction on $k$ that the subset $Q_k\subseteq Q$,
consisting of those products
$\overline{a_1}\cdots\overline{a_n}$ such that there are
precisely $k$ idempotents among $a_1,\ldots,a_n$, is finite. This
will then prove Step 3.

\textbf{Base of induction.}

$k=1$. To prove the base case, it is enough to show that
$\overline{a}$ is of finite order for all idempotents $a\in D$.

Let $a$ be an arbitrary idempotent from $D$. We have
$q(\overline{a}^n,x)=\overline{ax}^n$. If $ax\notin D$, then
$q_x\in\overline{I}\langle\overline{S}\rangle\cup\overline{I}$. Let
$ax\in D$. Consider $q_{x,y}=q(q_x,y)$. If $ax$ is not an idempotent
then
$q_{x,y}=\overline{axy}\cdot\overline{axaxy}\cdots\overline{(ax)^ny}\in
\overline{S}\cdot\overline{I}\cup\overline{S}\cdot\overline{I}\langle\overline{S}\rangle$.

Let $E$ be the set of all idempotents $\mathcal{R}$-equivalent to
$a$. Let $X_a$ be the set of all $x\in D$ such that $ax$ is an
idempotent in $D$. Notice that for $x\in D$, $ax\in D$ if and only
if $L_a\cap R_x$ is an idempotent.

Since the set
$\overline{I}\cup\overline{I}\langle\overline{S}\rangle\cup
\overline{S}\cdot\overline{I}\langle\overline{S}\rangle\cup\overline{S}\cdot\overline{I}$
is finite, we have that there are only finitely many elements
$q(\overline{a}^n,x)$ with $x\in S\setminus X_a$. On the other hand,
$q(\overline{a}^n,x)=\overline{ax}^n$ and $ax\in E$, for all $x\in
X_a$. Thus, by wreath recursions for all elements $\overline{a}^n$,
$a\in E$, we have that $\overline{a}$ is of finite order for every
idempotent $a\in D$. Hence the base case is established.

\textbf{Induction step.}

We will do step $k\mapsto k+1$. Take an arbitrary product
$\pi=\overline{a_1}\cdots\overline{a_n}\in Q_{k+1}$. There
are precisely $(k+1)$ different $\mathcal{R}$-classes among
$R_{a_1},\ldots,R_{a_n}$. Obviously, it would suffice to prove the
step if $a_1,\ldots,a_n$ come from fixed $(k+1)$
$\mathcal{R}$-classes (and for every of these $\mathcal{R}$-classes
there is at least one representative among $a_1,\ldots,a_n$). In
particular, in the remainder of the proof all the products from
$Q_{k+1}$ will involve these fixed $\mathcal{R}$-classes.

With every such product $\pi$ we associate the correspondent
$\mathcal{L}$-class $L(\pi)=L_{a_1}=\cdots=L_{a_n}$. We have
$q_x=q(\pi,x)=\overline{a_1x}\cdots\overline{a_nx}$ for
all $x\in S$. Notice that if $a_1x\in D$ then
$a_1x\mathcal{L}\cdots\mathcal{L}a_nx$ and $a_i\mathcal{R}a_ix$. In
addition, for every $\mathcal{L}$-class in $D$ there exists $x$ such
that $a_1x$ lies in this $\mathcal{L}$-class.

Now we split $S$ into three disjoint sets:
\begin{itemize}
\item
The set $A(\pi)$ of all $x$ such that $a_1x\notin D$.
\item
The set $B(\pi)$ of all $x$ such that $a_1x\in D$ and there are at
most $k$ idempotents among $a_1x,\ldots,a_nx$.
\item
The set $C(\pi)$ of all $x$ such that $a_1x\in D$ and there are
precisely $(k+1)$ idempotents among $a_1x,\ldots,a_nx$.
\end{itemize}
Notice that $a_1x\in D$ if and only if $L(\pi)\cap R_x$ is an
idempotent. Thus each of $A(\pi),B(\pi),C(\pi)$ depends only on
$L(\pi)$.

If $x\in A(\pi)$ then
$q_x\in\overline{I}\cup\overline{I}\langle\overline{S}\rangle$.

Let $x\in B(\pi)$. Take $y\in S$. We have
$q(q_x,y)=\overline{a_1xy}\cdots\overline{a_nxa_{n-1}x\cdots
a_1xy}$. Let $m$ be maximum such that $a_ix\cdots a_1xy\in D$ for
all $i\leq m$. Recall that $a_1x\mathcal{L}\cdots\mathcal{L}a_nx$.
So, as before, we have that $a_1x,\ldots,a_{m-1}x$ are idempotents.
There are at most $k$ such idempotents and so $a_1x,\ldots,a_{m-1}x$
split in at most $k$ $\mathcal{R}$-classes. We have $a_ix\cdots
a_1xy=a_ixy\mathcal{R}a_ix$ and so there are at most $k$ idempotents
among $a_1xy,\ldots,a_{m-1}xy$. Thus for all $y\in S$, we have
$q(q_x,y)\in\overline{I}\cup\overline{I}\langle\overline{S}\rangle\cup
Q_k\overline{S}(\overline{I}\cup\overline{I}\langle\overline{S}\rangle)$.
This implies that there are only finitely many $q_x$ for every
$\pi\in Q_{k+1}$ and $x\in B(\pi)$.

Let, finally, $x\in C(\pi)$. Since $a_1x,\ldots,a_nx$ lie in exactly
$(k+1)$ $\mathcal{R}$-classes, we have that all of
$a_1x,\ldots,a_nx$ are idempotents. In particular
$q_x=\overline{a_1x}\cdots\overline{a_nx}\in Q_{k+1}$ and
$a_1x,\ldots,a_nx$ involve the same (fixed) $\mathcal{R}$-classes as
$a_1,\ldots,a_n$. We also mention that if we fix some
$\mathcal{L}$-class $L$ in $D$ such that $L=L(\rho)$ for some
$\rho\in Q_{k+1}$, then the set of all $q(\pi,x)$, where $\pi\in
Q_{k+1}$, $L=L(\pi)$ and $x\in C(\pi)$, exhausts the whole of
$Q_{k+1}$.

Let $M$ be the total (finite) number of elements in $\{q(\pi,x):x\in
A(\pi)\cup B(\pi),~\pi\in Q_{k+1}\}$. Let also $N=M^{|S|}|S|$ and
$p$ be the number of $\mathcal{L}$-classes in $D$.

Take now any product
$\pi=\overline{a_1}\cdots\overline{a_{N^p+1}}\in Q_{k+1}$.
We will prove that $\pi$ equals some element from $Q_{k+1}$ of
length less than $N^p+1$. That will complete the induction step and
the whole proof of Step 3.

Let $L_1,\ldots,L_q$ be all $\mathcal{L}$-classes, which in
intersection with the fixed $\mathcal{R}$-classes give $(k+1)$
idempotents. Assume without loss of generality that $L(\pi)=L_1$.
Note that $\pi=q(\pi,x)$ for some $x\in C(\pi)$.

By Pigeonhole Principle, we have that there exist $1\leq
i_1<\cdots<i_{N^{p-1}+1}\leq N^{p}+1$ such that
$\tau(\overline{a_{1}}\cdots\overline{a_{i_j}})=
\tau(\overline{a_{1}}\cdots\overline{a_{i_k}})$ and
$q(\overline{a_{1}}\cdots\overline{a_{i_j}},x)=
q(\overline{a_{1}}\cdots\overline{a_{i_k}},x)$ for all
$x\in A(\pi)\cup B(\pi)$ and $j<k$. Notice now that
$L(\pi)=L(\overline{a_{i_1}}\cdots\overline{a_{i_k}})$.
There exists $y\in C(\pi)$ such that $L(q(\pi,y))=L_2$. Analogously,
by Pigeonhole Principle, we have that there is a subsequence
$i_1\leq j_1<\cdots<j_{N^{p-2}+1}\leq i_{N^{p-1}+1}$ such that
$\tau(\overline{a_{1}y}\cdots\overline{a_{j_u}y})=
\tau(\overline{a_{1}y}\cdots\overline{a_{j_v}y})$ and
$q(\overline{a_{1}y}\cdots\overline{a_{j_u}y},x)=
q(\overline{a_{1}y}\cdots\overline{a_{j_v}y},x)$ for all
$x\in A(q(\pi,y))\cup B(q(\pi,y))$ and $u<v$. Proceeding in this way
in total at most $q$ times we arrive at two indices $u<v$, such that
$$\tau(\overline{a_{1}z}\cdots\overline{a_{u}z})=
\tau(\overline{a_{1}z}\cdots\overline{a_{v}z})$$ and
$$q(\overline{a_{1}z}\cdots\overline{a_{u}z},x)=
q(\overline{a_{1}z}\cdots\overline{a_{v}z},x)$$ for all
$z\in C(\pi)$, $x\in A(q(\pi,z))\cup B(q(\pi,z))$.

Finally, we remark that if $x\in C(\pi)$ and $y\in C(q(\pi,x))$,
then $xy\in C(\pi)$. Thus, from wreath recursions for elements
$\overline{a_{1}z}\cdots\overline{a_{u}z}$ and
$\overline{a_{1}z}\cdots\overline{a_{v}z}$ for all $z\in
C(\pi)$, it now follows that
$\overline{a_{1}}\cdots\overline{a_{u}}=
\overline{a_{1}}\cdots\overline{a_{v}}$ and so
$$\pi=\overline{a_1}\cdots\overline{a_{u}}\cdot\overline{a_{{v}+1}}
\cdots\overline{a_{N^{p}+1}}$$ is of length strictly less than
$N^{p}+1$.

Therefore the induction step is proved and so Step 3 is established.

\vspace{\baselineskip}

{\bf\noindent Step 4: $\mathbf{\langle\overline{S}\rangle}$ is
finite.}

We have
\begin{equation*}
\langle\overline{S}\rangle=\overline{I}\langle\overline{S}\rangle^1
\cup\langle\overline{S}\setminus\overline{I}\rangle\cup\langle\overline{S}\setminus\overline{I}\rangle
\overline{I}\langle\overline{S}\rangle^1
\end{equation*}
is finite by Steps 2 and 3.
\end{proof}

\section{Proof of Proposition~\ref{pr:free}}\label{sec:free}

\begin{proof}[Proof of Proposition~\ref{pr:free}]
$(\Rightarrow)$. Suppose that $\mathbf{C}(S)$ is free. Let $K$ be
the minimal ideal of $S$. Then $K$ is a Rees matrix semigroup
$\mathcal{M}[G;I,J;P]$ for some $J\times I$-matrix $P$ and group $G$
with identity $e$.  By~\cite[Theorem~3.4.2]{H} we even may assume
that $1\in I$, $1\in J$ and $p_{j1}=p_{1i}=e$ for all $i\in I$,
$j\in J$. Then the element $k=(1,e,1)\in K$ is clearly an
idempotent. Then $sk=sk^2$ and $sk\in I$ for all $s\in S$.
Therefore, by wreath recursions,
$\overline{k}\cdot\overline{s}=\overline{k}\cdot\overline{sk}$.
Since a free semigroup is left cancellative, we obtain
$\overline{s}=\overline{sk}$ and so, by Lemma~\ref{lm:lambda},
$\lambda_{s}=\lambda_{sk}$. Hence $\overline{S}$ coincides with
$\overline{L_k}$, where $L_k$ is the $\mathcal{L}$-class containing
$k$. Let $j\in J$ and $i\in I$. The condition
$\lambda_{(1,e,j)}=\lambda_{(1,e,j)(1,e,1)}$ implies that (since
$p_{j1}=p_{1i}=e$)
\begin{equation*}
(1,p_{ji},1)=(1,e,j)(i,e,1)=(1,e,j)(1,e,1)(i,e,1)=(1,e,1),
\end{equation*}
and so $p_{ji}=e$. Then for all $i,h\in I$ we have
$\overline{(i,e,1)}\cdot\overline{(h,e,1)}=\overline{(j,e,1)}\cdot\overline{(h,e,1)}$.
It follows that $\overline{(i,e,1)}=\overline{(h,e,1)}$ and hence
$\lambda_{(i,e,1)}=\lambda_{(h,e,1)}$. Thus $i=h$ and so $K$
contains only one $\mathcal{R}$-class.

Finally, since $\lambda_{s}=\lambda_{sk}$ for all $s\in S$, we have
that $S^2\subseteq K$. Hence the only non-singleton
$\mathcal{H}$-classes in $S$ must be those lying in $K$. If $K$
contains singleton $\mathcal{H}$-classes, then $S$ is
$\mathcal{H}$-trivial and so, by Theorem~\ref{th:finite},
$\mathbf{C}(S)$ is finite, a contradiction. Thus all
$\mathcal{H}$-classes in $K$ are non-singleton.

$(\Leftarrow)$. Since $K$ contains only one $\mathcal{R}$-class, we
have that $K=G\times R$ where $G$ is a group with identity $e$ and
$R$ is a right zero semigroup. Let $k=(h,s)\in K$ be as in the
hypothesis. Then for every $(g,r)\in K$, we have
$\lambda_{(g,r)}=\lambda_{(g,r)(h,s)}=\lambda_{(gh,s)}$. Then
$(g,t)=(g,r)(e,t)=(gh,s)(e,t)=(gh,t)$ and therefore $h=e$ and
$k=(e,s)$. Hence, by Lemma~\ref{lm:lambda},
$\overline{S}=\overline{H_{(e,s)}}$. As in the proof of
Theorem~\ref{th:finite}, we have that $\mathbf{C}(S)$ can be
homomorphically mapped onto $\mathbf{C}(H_{(e,s)})$. But by
Theorem~\ref{th:ss-groups}, $\mathbf{C}(H_{(e,s)})$ is free of rank
$|H_{(e,s)}|$, so $\mathbf{C}(S)$ is free of rank $|H_{(e,s)}|$.
\end{proof}

\section{Proof of Propositions~\ref{pr:right-zero}
and~\ref{pr:left-zero}}\label{sec:left-right}

\begin{proof}[Proof of Proposition~\ref{pr:right-zero}]
$(\Rightarrow)$. Suppose that $\mathbf{C}(S)$ is a right zero
semigroup. Let $a,b\in S$. Then
$\overline{b}\cdot\overline{a}=\overline{a}$. In particular,
$\lambda_{ab}=\lambda_{a}$. This implies that $abc=ac$ for all
$a,b,c\in S$.

$(\Leftarrow)$. Suppose that $abc=ac$ for all $a,b,c\in S$. Then
$\lambda_{ab}=\lambda_a$ for all $a,b\in S$. Now,
$\overline{b}\cdot\overline{a}=\overline{a}$ if and only if
$\lambda_{ab}=\lambda_a$ and
$\overline{bx}\cdot\overline{abx}=\overline{ax}$ for all $x\in S$.
By hypothesis, the latter is equivalent to
$\overline{bx}\cdot\overline{ax}=\overline{ax}$. By recursive
arguments we now obtain that
$\overline{b}\cdot\overline{a}=\overline{a}$ for all $a,b\in S$.
Thus $\mathbf{C}(S)$ is a right zero semigroup.
\end{proof}

\begin{proof}[Proof of Proposition~\ref{pr:left-zero}]
$(\Rightarrow)$. Suppose that $\mathbf{C}(S)$ is a left zero
semigroup. Since this left zero semigroup is finitely generated, it
is finite. So, by Theorem~\ref{th:finite}, $S$ is
$\mathcal{H}$-trivial. Let $I$ be the minimal ideal in $S$. Then
$\langle\overline{I}\rangle\subseteq \mathbf{C}(S)$ can be
homomorphically mapped onto $\mathbf{C}(I)$. Since $I$ is simple and
finite, it is a Rees matrix semigroup. Since $S$ is
$\mathcal{H}$-trivial, $I=X\times Y$ for some left zero semigroup
$X$ and a right zero semigroup $Y$. By
Lemmas~\ref{lm:left-zero} and~\ref{lm:direct},
$\mathbf{C}(I)$ is a right zero semigroup on $|X|$ points. A
homomorphic image of the left zero semigroup
$\langle\overline{I}\rangle$ must be a left zero semigroup. Hence
$|X|=1$ and so $I$ is a right zero semigroup.

Let $s\in S$ and $i\in I$. Then, since $\mathbf{C}(S)$ is a left
zero semigroup, $\overline{s}\cdot\overline{i}=\overline{s}$;
consequently $\lambda_{s}=\lambda_{is}$. By Lemma~\ref{lm:lambda},
$\overline{s}=\overline{is}\in\overline{I}$. In particular
$S\lambda_{s}\subseteq I$. Since this holds for each $s\in S$, it
follows that $S^2\subseteq I$. Thus $I=S^2$.

$(\Leftarrow)$. Suppose that the minimal ideal $I$ of $S$ coincides
with $S^2$ and that $I$ is a right zero semigroup.

Take an arbitrary $s\in S$ and fix $i\in I$. Then for every $x\in S$
we have that $sx\in I$ and so $isx=sx$. This implies that
$\lambda_{s}=\lambda_{is}$. By Lemma~\ref{lm:lambda}, we have that
$\overline{s}=\overline{is}$. Therefore $\overline{S}=\overline{I}$
and in particular $\mathbf{C}(S)=\langle\overline{I}\rangle$. So it
suffices to prove that $\overline{i}\cdot\overline{j}=\overline{i}$
for all $i,j\in I$. Note that if $\alpha\in S^{\infty}$, then
$\alpha\cdot\overline{i}\in I^{\infty}$. Since $I$ is a right zero
semigroup, $\overline{j}$ acts identically on $I^{\infty}$. Hence
$\alpha\cdot(\overline{i}\cdot\overline{j})=\alpha\cdot\overline{i}$
for all $\alpha\in S^{\infty}$ and so
$\overline{i}\cdot\overline{j}=\overline{i}$, as required.
\end{proof}

\section{Further Discussion}\label{sec:further}

In Theorem~\ref{th:whenisgroup} we proved that no Cayley automaton
semigroup can be a non-trivial group. In addition, it is proved
in~\cite{Mi} that if $S$ is a finite $\mathcal{H}$-trivial monoid,
then $\mathbf{C}(S)$ is a (finite) $\mathcal{H}$-trivial semigroup.
In fact, the author believes that every Cayley automaton semigroup
is $\mathcal{H}$-trivial and poses an

\begin{problem}\label{pr:aperiodic}
Are all Cayley automaton semigroups $\mathcal{H}$-trivial?
\end{problem}

The following proposition is an important consequence of
Theorem~\ref{th:finite}.

\begin{proposition}\label{pr:important}
Any infinite Cayley automaton semigroup contains a free semigroup of
rank $2$.
\end{proposition}

\begin{proof}
Suppose $\mathbf{C}(S)$ is infinite. Then $S$ is not
$\mathcal{H}$-trivial. So $S$ contains an $\mathcal{H}$-class $H$
with $|H|>1$. Then as in the proof of the necessity of
Theorem~\ref{th:finite}, there exists a subsemigroup $T\leq S$ such
that $\langle\overline{T}\rangle$ has a free semigroup of rank $|H|$
as a homomorphic image.
\end{proof}

\begin{corollary}\label{cor:nice}
The free product of two trivial semigroups $\mathrm{Sg}\langle
e,f\mid e^2=e,~f^2=f\rangle$ and free commutative semigroups of rank
$>1$ are all automaton semigroups, but neither of them is a Cayley
automaton semigroup.
\end{corollary}

\begin{proof}
That free commutative semigroups of rank $>1$ are automaton
semigroups can be found in~\cite{M}; and it is routine to check that
the automaton semigroup generated by the automaton

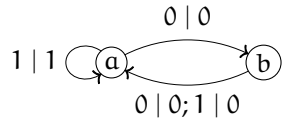
\begin{figure}[htp]
\centering
\includegraphics{cm_1auto.5}
\end{figure}
is isomorphic to $\mathrm{Sg}\langle e,f\mid e^2=e,~f^2=f\rangle$.
That neither of these semigroups is a Cayley automaton semigroup
follows immediately from Proposition~\ref{pr:important}.
\end{proof}

\begin{remark}
The characterization of those finite semigroups $S$ such that
$\mathbf{C}(S)$ is a right zero semigroup, is `close' to the
characterization of rectangular bands: the latter are precisely
those semigroups $S$ such that all the elements from $S$ are
idempotents and $abc=ac$ for all $a,b,c\in
S$,~\cite[Theorem~1.1.3]{H}.
\end{remark}

In the following example we show that it is possible for Cayley
automaton semigroup to be a \emph{non-trivial} left zero semigroup:

\begin{example}
Define a finite semigroup $S$ on four points $i$, $j$, $k$, $f$ with
the following multiplication table:
\begin{equation*}
\begin{array}{c|cccc}
& i & j & k & f\\
\hline
i & i & j & k & i \\
j & i & j & k & i \\
k & i & j & k & j \\
f & i & j & k & i
\end{array}
\end{equation*}
Then $\mathbf{C}(S)$ is a left zero semigroup on $2$ points.
\end{example}

\begin{proof}
One checks that the multiplication table indeed gives a semigroup.
By Lemma~\ref{lm:lambda}, $\overline{i}=\overline{j}=\overline{f}$.
Hence, by Proposition~\ref{pr:left-zero}, $\mathbf{C}(S)$ is a left
zero semigroup generated by $\overline{j}$ and $\overline{k}$. It
remains to notice that $\overline{j}\neq\overline{k}$. It follows
from Lemma~\ref{lm:lambda} and $f\lambda_{j}=jf=i\neq
j=kf=f\lambda_{k}$.
\end{proof}


\end{document}